\documentclass[12pt,oneside]{amsart}

 \usepackage{amssymb}
\theoremstyle{plain}  
      \newtheorem{theorem}{Theorem}[section]
      \newtheorem{lemma}[theorem]{Lemma}
      \newtheorem{corollary}[theorem]{Corollary}
      \newtheorem{proposition}[theorem]{Proposition}
      \theoremstyle{definition}
      \newtheorem{definition}[theorem]{Definition}
       \newtheorem{example}{Example}[section] 
      \theoremstyle{remark}

\begin{document}
\author{N. Ghroda}
  \address{Department of Mathematics\\University
  of York\\Heslington\\York YO10 5DD\\UK}
\email{ng521@york.ac.uk}




   \title[Primitive inverse semigroups of I-quotients]{Primitive inverse semigroups of left I-quotients}

   \begin{abstract}
     A subsemigroup  $S$  of an inverse semigroup  $ Q $  is a left I-order in  $Q$,  if every element in  $Q$  can be written as  $a^{-1}b$    where  $a ,b \in S$  and  $a^{-1}$  is the inverse of $a$  in the sense of inverse semigroup theory.  We study a characterisation of semigroups which have a primitive inverse semigroup of left I-quotients.
   \end{abstract}

   \keywords{ primitive inverse semigroup , I-quotients, I-order}

   \date{\today}

   \maketitle
 
\section{Introduction}

Clifford \cite{clifford} showed that from any right cancellative monoid  $S$ with (LC) Condition, there is a bisimple inverse monoid  $Q$ such that  $Q=S^{-1}S$, that is, every element  $q$  in  $Q$ can be written as  $a^{-1}b$    where  $a ,b \in S$. By saying that a semigroup  $S$ has the (LC) Condition we mean for any  $a,b\in S$ there is an element  $c\in S$ such that  $Sa\cap Sb=Sc$. In \cite{GG} the authors have  extended Clifford's work to a left ample semigroup with (LC) where they  introduced the following definition of left I-order in inverse semigroups.\\
 \\
 Let  $Q$ be an inverse semigroup. A  subsemigroup  $S$  of  $Q$  is a \emph{left I-order} in  $Q$ and $Q$ is a \emph{semigroup of left I-quotient} of $S$, if every element in  $Q$  can be written as  $a^{-1}b$    where  $a ,b \in S$  and  $a^{-1}$  is the inverse of  $a$  in the sense of inverse semigroup theory. \emph{Right I-order} and \emph{semigroup of right I-quotients} is defined dually. If  $S$ is both a left and a right I-order in an inverse semigroup  $Q$, we say that  $S$  is an \emph{I-order} in  $Q$ and $Q$ is a semigroup of \emph{I-quotients} of $S$.
 This notion extends the classical notion of left order in an inverse semigroup  $Q$, due to Fountain and Petrich \cite{pjhon}. We say that  a subsemigroup  $S$  of a semigroup $Q$  is a \emph{left order} in  $Q$   or  $Q$  is a \emph{semigroup of left quotients} of $S$ if every element of   $ Q $  can be written as   $a^{\sharp}b$  where  $a  , b \in S$  and  $a^{\sharp}$  is the inverse of   $a$  in a subgroup of  $Q$ and if, in addition, every square-cancellable element  (an element  $a$ of a semigroup  $S$ is square-cancellable if  $a \mathcal{H}^{*}a^{2}$) lies in a subgroup of  $Q$.\\
\\
  Clearly if  $S$ has an inverse semigroup of left quotients  $Q$, then  $Q$ is also a semigroup of left I-quotients, but the converse is not true as will see by an example. \emph{Right order} and \emph{semigroup of right quotients} are defined dually. If  $S$ is both a left and right order in  $Q$, then  $S$  is an \emph{order }in  $Q$ and  $Q$ is a\emph{ semigroup of quotients} of  $S$. \\
\\
In this article we focus on studying left I-orders in primitive inverse semigroups.\\
\\ 
 In Section~\ref{sec:prelim} we begin by investigating some properties of semigroups which are left I-orders in  primitive inverse semigroups. The next section is devoted to the proof of Theorem~\ref{maintheorem} which characterizes those semigroups which have a primitive inverse semigroup of left I-quotients. We then specialise our result to left I-orders in Brandt semigroups a result which may be regarded as a generalisation of the main theorem in \cite{fountain}, which characterised left orders in Brandt semigroups. The statement
of an alternative description of left I-orders in Brandt semigroups has been
privately communicated by Cegarra \cite{cp}. \\
  In Section~\ref{sec:UBrandt} we show that a primitive inverse semigroup of left I-quotients is  unique up to isomorphism. Section~\ref{sec:I-order}  then concentrates on  I-orders (two-sided case) in primitive inverse semigroups. In the final Section we give characterizations of adequate (ample) semigroups having primitive inverse semigroups of left I-quotients.

   \section{Preliminaries}
   \label{sec:prelim} 
Throughout this article, we shall follow the terminologies and notation of \cite{clifford}. The set of non-zero elements of a semigroup  $S$ will be denoted by  $S^{*}$.\\   
The relations $\mathcal{R^{*}},  \mathcal{L^{*}}$ and  $\mathcal{H^{*}}$ play a significant role in this article. It is well known that  the relation   $\mathcal{R^{*}}$ is defined on a semigroup  $S$ by the rule that  $a\,\mathcal{R^{*}}\,b$ in  $S$ if  $a\,\mathcal{R}\,b$ in some oversemigroup $T$ of  $S$, and this equivalent to 
 $a\,\mathcal{R^{*}}\,b$ if and only if  \[ xa=ya  \quad \mbox{if and only if} \quad xb=yb\] for all  $x,y \in S^{1}$. The relation  $\mathcal{L^{*}}$ is defined dually and  $\mathcal{H^{*}}=\mathcal{R^{*}}\cap  \mathcal{L^{*}}$. It is clear that  $\mathcal{R} \subseteq  \mathcal{R^{*}}$ and  $\mathcal{L} \subseteq  \mathcal{L^{*}}$ where  $\mathcal{R}$ and  $\mathcal{L}$ are the usual Green's relations. \\

We recall that a semigroup  $S$ with zero is defined to be \emph{categorical at  0}  if whenever  $a,b,c \in S$ are such that  $ab\neq 0$ and  $ bc \neq 0$, then  $abc\neq 0$. We say that,  $S$ is \emph{0-cancellative} if  $b=c$ follows from  $ab=ac\neq 0$  and from  $ba=ca\neq0$. \\
\\
An  inverse semigroup   $S$ with zero  is a \emph{primitive inverse} semigroup if all its nonzero
idempotents are primitive, where an idempotent  $e$ of  $S$ is called \emph{primitive} if  $e\neq 0$ and  $f\leq e$ implies  $f=0$ or  $e=f$. We will use the following facts about primitive inverse semigroups heavily through this article.

\begin{lemma}\label{Gjhon}\cite{Gjhon}
Let  $Q$ be a primitive inverse semigroup.\\
$(i)$ \ $Q$ is categorical at 0.\\
$(ii)$ \ If  $e,f \in E^*(Q)$, then  $ef\neq 0$ implies  $e=f$.\\
$(iii)$ \ If  $e\in E^*$ and  $s\in Q^*$, then \[es\neq 0 \quad \mbox{implies} \quad es=s \quad \mbox{and} \quad se\neq 0 \quad \mbox{implies} \quad se=s.\]
$(iv)$ \ If $a,s \in Q^*$ and  $as=a$, then  $s=a^{-1}a$. Dually, if $sa=a$, then  $s=aa^{-1}$\\
$(v)$\ If $ab\neq 0$, then $a^{-1}a=b^{-1}b$.\end{lemma}
 From the above lemma we can notice easily that a primitive inverse semigroup is 0-cancellative.\\
To investigate the properties of a semigroup  $S$ which is a left I-order in a primitive inverse semigroup  $Q$ we need the relations  $\lambda\ , \rho$ and  $\tau$  which are introduced in \cite{Lall} on any semigroup with zero as follows:\\
\[a\,\lambda\,b \ \mbox{if and only if} \ a=b=0 \ \mbox{or}\ Sa\cap Sb\neq 0 \]
\[a\,\rho\,b \ \mbox{if and only if} \ a=b=0 \ \mbox{or}\ aS\cap bS\neq 0 \]
\[\tau=\rho \cap \lambda.\]
\begin{lemma}\label{conc}
 Let $S$ be any semigroup with zero. \\
 $(i)$ If $a\,\mathcal{R^*}\,b$ where $a,b\neq 0$, then \[xa\neq 0 \quad \mbox{if and only if}\quad xb\neq 0.\]
 $(ii)$ If $S$ is categorical at 0 and  0-cancellative, then \[x,xa\neq 0 \quad \mbox{implies that }\quad  x\,\mathcal{R^*}\,xa,\] for any $x,a \in S$.\end{lemma}
\begin{proof}$(i)$ \ Clear. \\
$(ii)$ Let $x,a \in S$ with $xa\neq 0$. If $u,v \in S^1$ and   $ux=vx$, then clearly $uxa=vxa$. \\
\\
Conversely, if $uxa=vxa\neq 0$, then by 0-cancellativity, $ux=vx\neq 0$. On the other hand, if $uxa=vxa=0$, then 
by categoricity at 0, $ux=vx=0$ (note that in this case, $u,v \neq 1$).\end{proof}
\begin{definition}
Let  $S$  be a subsemigroup of an inverse semigroup  $Q$. Then  $S$ is a \emph{straight} left I-order in  $Q$  if, every  $q \in Q$   can be write as $q=a^{-1}b$, where  $a,b \in S$  and   $a$  and  $b$  are  $\mathcal{R}$-related \ in  $Q$.
\end{definition}
In the next lemma we introduce some properties of a semigroup which has a primitive inverse semigroup of left I-quotients. \\
\\
We made the convention that if  $S$ is a left I-order in  $Q$, then  $\mathcal{R}$ and  $\mathcal{L}$ will be relations on  $Q$ and  $\mathcal{R^*}, \mathcal{L^*},\lambda, \rho$ and  $\tau$ will refer to  $S$.
\begin{proposition}\label{Inthenameofgod}
Let \ $S$  be a subsemigroup of a primitive inverse semigroup  $Q$. If  $S$  is left I-order in  $Q$, then: \\
$1)$ \ $S$ contains the  0  element of  $Q$;\\
$2)$ \ $ \mathcal{L} \cap (S \times S) =\lambda$;\\
$3)$ \ $S$  is a straight left I-order in  $Q$; \\
$4)$ \ $Sa\neq 0$  for all  $a \in S^{*}$; \\ 
$5)$ \ $\mathcal{R}\cap(S\times S)=\mathcal{R}^{*}$; \\
$6)$ \ $\rho \subseteq \mathcal{R^*}$.
 \end{proposition}
\begin{proof}$1)$ \ If  $S \subseteq H_{e}$ for some  $0 \neq e \in E(Q)$, then  $0=a^{-1}b \in H_{e}$  which is a contradiction, so  $S\not\subseteq H_{e}$. Thus  $ 0\in S$  or there exist  $(i,a,j) \in S,  i \neq j$. Then  $(i,a,j)(i,a,j)=0$, so that $0 \in S$.\\
\\
$2)$\ If  $ a\,\lambda\,b$, then  $a=b=0$ and certainly   $ a\,\mathcal{L}\,b $  in  $Q$, or  $xa=yb\neq 0$  for some  $ x , y \in S$. In the latter case,  $a=x^{-1}yb,  b=y^{-1}xa$, so that $ a\,\mathcal{L}\,b $  in  $Q$. Conversely, if  $ a\,\mathcal{L}\,b $  in  $Q$, then either  $a=b=0$, or $a\neq 0$ and  $a=x^{-1}yb $  for some  $x^{-1}y \in Q$  where  $ x , y \in S$. Then  $ xa=yb\neq 0$. Hence  $a\,\lambda\,b $  in  $S$. It is worth pointing out that in this case  $\lambda$ is transitive. Moreover, it is an equivalence, and  $\{0\}$ is a  $\lambda$-class.\\
\\
$3)$ \  Suppose that  $0 \neq q \in Q$, then  $q=a^{-1}b$  for some  $a ,b \in S$. Since  $a^{-1}b\neq 0$, Lemma ~\ref{Gjhon} gives $aa^{-1}=bb^{-1}$ so that $a\,\mathcal{R}\,b$ in $Q$. By categoricity at  0  in  $Q$ and  Lemma~\ref{Gjhon},  $0 \neq aa^{-1}b=b$ and  $0 \neq a^{-1}bb^{-1}=a^{-1}$, then  $aa^{-1}=aa^{-1}bb^{-1}=bb^{-1}$ in  $Q$. Thus  $a\,\mathcal{R}\,b$  in  $Q$.\\
\\ 
4) \ Let  $a=x^{-1}y\neq 0$  for some  $x , y \in S$, where  $x\, \mathcal{R}\,y $ in  $Q$. By categoricity at  0  and Lemma~\ref{Gjhon}   we have  $xa=y\neq 0$. So  $Sa \neq 0.$ \\
\\ 
5)\ It is clear that  $\mathcal{R}\cap(S\times S)\subseteq \mathcal{R}^{*}$. To show that  $\mathcal{R^*} \subseteq \mathcal{R}\cap(S\times S)$. Let  $ a\,\mathcal{R}^{*}\,b $ \ in \ $S$; from (4) there exist  $y$  in  $S$  such that  $ya\neq 0$, by Lemma~\ref{conc}  $yb\neq 0$. Then by Lemma~\ref{Gjhon} we have \ $aa^{-1}=y^{-1}y=bb^{-1}$  and we get  $ a\,\mathcal{R}\,b $  in  $Q$. \\
 \\
6) \ Suppose that \ $a\,\rho\,b$\ in \ $S$, then \ $a=b=0$\ and \ $a\,\mathcal{R}\,b$\ in \ $Q$, or\ $ax=by\neq 0$\ for some \ $x,y \in S$. Then \ $b=axy^{-1} , a=byx^{-1}$, so that    \ $a\,\mathcal{R}\,b$\ in \ $Q$. By (5)   \ $a\,\mathcal{R^{*}}\,b$\ in \ $S$. 

 \end{proof}

By Lemma~\ref{Gjhon} and Proposition~\ref{Inthenameofgod} the following corollaries are clear.

\begin{corollary}\label{pslio}Let  $S$ be a left I-order in a primitive inverse semigroup  $Q$. If  $a^{-1}b\neq 0$, then  $a\,\mathcal{R}\,b$. Consequently,  $a^{-1}\,\mathcal{R}\,a^{-1}b\,\mathcal{L}\,b$.\end{corollary}

\begin{corollary}\label{prop}
Let  $S$ be a left I-order in a primitive inverse semigroup  $Q$. Let  $a^{-1}b , c^{-1}d$ be non-zero elements of \ $Q$\ where  $a,b,c$ and  $d$ are in  $S$. Then \\
$1)$\ $a^{-1}b\,\mathcal{R}\,c^{-1}d$  if and only if  $a\,\lambda\,c$; \\
$2)$ \ $a^{-1}b\,\mathcal{L}\,c^{-1}d$  if and only if   $b\,\lambda\,d$. \end{corollary}
\begin{proof} $1)$ \ We have that  $a^{-1}b\,\mathcal{R}\,c^{-1}d$ if and only if  $a^{-1}\,\mathcal{R}\,c^{-1}$ if and only if  $a\,\mathcal{L}\,c$. By Proposition~\ref{Inthenameofgod} this is equivalent to  \ $a\,\lambda\,c$. \\
$2)$\ Similar.\end{proof}

   \section{The main theorem}
   \label{main}

  The aim of this section is to prove the following theorem;
\begin{theorem}\label{maintheorem}
A semigroup  $S$ is a left I-order in a primitive inverse semigroup  $Q$  \textit{if and only if }  $S$  satisfies the following conditions:
\\
$(A)$ \ $S$  is categorical at  $0$; \\
$(B)$ \ $S$  is 0-cancellative;\\
$(C)$ \ $\lambda$  is transitive; \\
$(D)$ \ $Sa\neq 0$  for all  $a\in S$. \end{theorem}
\begin{proof}Suppose that  $Q$ is exists, then  $S$  inherits Conditions (A) and (B) from  $Q$. By Proposition~\ref{Inthenameofgod} first we have that Conditions (C) and (D) hold.\\

Conversely, suppose that  $S$ satisfies Conditions (A)-(D). Our aim now is to construct a semigroup  $Q$ in which if  $S$ is embedded as a left I-order in  $Q$. We remark that from (C), $\lambda$ is an equivalence and from the definition of $\lambda$, $\{0\}$ is a $\lambda$-class. Let  \[\Sigma =\{(a,b) \in S \times S : a\,\mathcal{R}^{*}\,b\},\] and 
\[\Sigma^* =\{(a,b) \in \Sigma: a,b\neq 0\}.\]
On  $\Sigma$ define $\sim$  as follows; \\
\ $(a,b)\sim(c,d) \Longleftrightarrow  \ a=b=c=d=0$, or there exist   $x , y \in S^{*}$  such that  \[ xa=yc \neq 0 , \ xb=yd \neq 0. \]
\begin{lemma}\label{equivalence}
$ \sim$ is an equivalence. \end{lemma}
\begin{proof}
It is clear that  $\sim$  is  symmetric. If  $(a,b)\in \Sigma^*$, by (D) there exist  $h\in S$ such that  $ha\neq 0$\ and hence as $a\,\mathcal{R^*}\,b$, $hb\neq 0$ and so that  $\sim$ is reflexive.  Let \[(a,b) \sim (c,d)\sim (p,q),\] where   $(a,b),(c,d)$ and  $(p,q)$ in  $\Sigma^*$.  Then there exist  $ x , y ,\bar{x} , \bar{y} $  such that  \[ xa=yc \neq 0  ,  xb=yd \neq 0  \quad  \mbox{and}  \quad \bar{x}c=\bar{y}p \neq 0  , \bar{x}d=\bar{y}q \neq 0.\]  To show that  $\sim$ is transitive, we have to show that, there are  $z , \bar{z} \in S$ such that  $ za=\bar{z}p \neq 0  , zb=\bar{z}q \neq 0$. \\ 
\\
Now,  $yc\,\lambda\,\bar{x}c$. For  $Sc \neq 0$  and  $Syc \neq 0$   and clearly  $ Sc \cap Syc \neq 0$, so that  $ c\,\lambda\,yc $. Similarly,  $\bar{x}c\,\lambda\,c$; since  $\lambda$  is transitive, we obtain  $yc\,\lambda\,\bar{x}c$. Hence 
$wyc=\bar{w}\bar{x}c \neq 0$. Thus  
$wxa=wyc=\bar{w}\bar{x}c=\bar{w}\bar{y}p\neq 0$, that is,  $wxa=\bar{w}\bar{y}p \neq 0$. As  $c\,\mathcal{R^*}\,d$ we have that  $wyd=\bar{w}\bar{x}d\neq 0$ so that similarly,  $wxb=\bar{w}\bar{y}q \neq 0$ as required. \end{proof}

Let  $[a,b]$\ denote the  $\sim$-equivalence class of  $(a,b)$. We stress that  $[0,0]$ contains only the pair  $(0,0)$. On  $Q=\Sigma/\sim$  define a multiplication as:
\begin{displaymath}
[a,b][c,d] = \left\{ \begin{array}{ll}
[xa,yd] & \textrm{if $b\,\lambda\,c$ \; \mbox{and} \;$xb=yc\neq 0$}\\
 \ 0 & \textrm{else}
\end{array} \right.\end{displaymath}
and \ $0[a,b]=[a,b]0=00=0$, where \ $0=[0,0]$. We put $Q^*=Q\setminus\{[0,0]\}$\\
\\
Before we show that the above multiplication is well-defined we can notice easily that  $ [xa,yd] \in Q$. For   $ xa\,\mathcal{R}^{*}\,xb=yc\,\mathcal{R}^{*}\,yd$.
\begin{lemma}\label{welldefined}
The multiplication is well-defined. \end{lemma}
\begin{proof}Suppose that  $ [a_{1},b_{1}]=[a_{2},b_{2}], [c_{1},d_{1}]=[c_{2},d_{2}]$ are in  $Q^{*}$. Then there are elements  $x_{1} , x_{2} , y_{1} , y_{2}$ in  $S$ such that
\[\begin{array}{rcl} x_{1}a_{1}=x_{2}a_{2}\neq 0,\\
x_{1}b_{1}=x_{2}b_{2}\neq 0,\\
y_{1}c_{1}=y_{2}c_{2}\neq 0, \\ y_{1}d_{1}=y_{2}d_{2}\neq 0.\end{array}\]
Now, 
\begin{displaymath}
[a_{1},b_{1}][c_{1},d_{1}] = \left\{ \begin{array}{ll}
[wa_{1},\bar{w}d_{1}] & \textrm{if \; $b_{1}\,\lambda\,c_{1}$} \; \mbox{and} \;wb_{1}=\bar{w}c_{1}\neq 0\\
0 & \textrm{else}
\end{array} \right.
\end{displaymath}
and
\begin{displaymath}
[a_{2},b_{2}][c_{2},d_{2}] = \left\{ \begin{array}{ll}
[za_{2},\bar{z}d_{2}] & \textrm{if \; $b_{2}\,\lambda\,c_{2}$} \; \mbox{and} \; zb_{2}=\bar{z}c_{2}\neq 0\\
0 & \textrm{else}.
\end{array} \right.
\end{displaymath}
Notice that as  $b_1\,\lambda\,b_2$ and  $c_1\,\lambda\,c_2$, we have that  $b_1\,\lambda\,c_1$ if and only if  $b_2\,\lambda\,c_2$. Hence  $[a_1,b_1][c_1,d_1]=0$ if and only if  $[a_2,b_2][c_2,d_2]=0$. We now assume that  $b_1\,\lambda\,c_1$ (and so  $b_2\,\lambda\,c_2$ also).\\
\\
 We have to prove that  $[wa_{1},\bar{w}d_{1}]=[za_{2},\bar{z}d_{2}]$ that is, \[xwa_{1}=yza_{2} \neq 0, \quad  x\bar{w}d_{1}=y\bar{z}d_{2}\neq 0, \; \mbox{ for some} \;  x , y \in S. \]
Since  $wb_{1}=\bar{w}c_{1}\neq 0  , zb_{2}=\bar{z}c_{2} \neq 0 $ and  $a_1\,\mathcal{R^*}\,b_1, a_2\,\mathcal{R^*}\,b_1$ we have\ $wa_{1} \neq 0$\ and \ $za_{2} \neq 0$. Hence 
\[ wa_{1}\, \lambda\,a_{1}\, \lambda\,x_{1}a_{1}=x_{2}a_{2} \,\lambda\,a_{2}\, \lambda\,za_{2}.\] 
By (C)  $wa_{1}\,\lambda\,za_{2}$, that is,  $xwa_{1}=yza_{2} \neq 0$ for some  $x , y \in S$.\\
\\
The following lemma is essentially Lemma 4.8 in \cite{fountain}. We give it for completeness.
\begin{lemma}\label{thesamelemma}
Let  $a,b,c,d,s,t,x,y$ be non-zero elements of  $S$ which satisfy
\[sa=tc \neq 0 , \ sb=td \neq 0 , \ xa=yc \neq 0 .\] Then \ $ xb=yd \neq 0$.\end{lemma}
\begin{proof}
Since  $sa\neq 0 , xa\neq 0$ and  $ sa\,\lambda\,a\,\lambda\,xa$, then  $sa\,\lambda\,xa$ that is, there are elements \ $w,z \in S$ such that  $zsa=wxa\neq 0$. Since  $S$ is 0-cancellative  and $a\neq 0$ we have  $zs=wx\neq 0$. Thus by categoricity at 0  \[ztc=zsa=wxa=wyc\neq 0.\] Cancelling  $c$ gives  $zt=wy\neq 0$. By categoricity at 0 again we have \[wxb =zsb =ztd =wyd\neq 0.\]Hence  $xb=yd\neq 0$.\end{proof}
We can apply it as follows: since  $ x_{1}a_{1}=x_{2}a_{2} \neq 0  , x_{1}b_{1}=x_{1}b_{2} \neq 0$ and 
    $xwa_{1}=yza_{2}$, then  $ xwb_{1}=yzb_{2} \neq 0$.  Now,  $wb_{1}=\bar{w}c_{1}  , zb_{2}=\bar{z}c_{2}$, then  \[x\bar{w}c_{1}=xwb_1=yzb_2= y\bar{z}c_{2} \neq 0.\]  
Reapply the same lemma to get  $x\bar{w}d_{1}=y\bar{z}d_{2}\neq 0$. \end{proof}
\begin{lemma}\label{assocciative}
The multiplication is associative. \end{lemma}
\begin{proof}
Let  $[a,b],[c,d],[p,q]\in Q^{*}$ and set
\begin{displaymath} 
X=([a,b][c,d])[p,q] = \left\{ \begin{array}{ll}
[xa,yd][p,q] & \textrm{if $b\,\lambda\,c$ \; \mbox{and} \;$xb=yc\neq 0$}\\
0 & \textrm{else}
\end{array} \right.\end{displaymath} 
and
\begin{displaymath}
 Y=[a,b]([c,d][p,q]) = \left\{ \begin{array}{ll}
[a,b][\bar{x}c,\bar{y}q] & \textrm{if $d\,\lambda\,p$ \; \mbox{and} \;$\bar{x}d=\bar{y}p\neq 0$}\\
0 & \textrm{else.}
\end{array} \right.\end{displaymath}
Suppose that  $X = 0$. If  $b\not\hspace{.05mm}\lambda\,c$, then either  $d\not\hspace{.05mm}\lambda\,p$ (in which case  $Y=0$) or,  $d\,\lambda\,p$ and  $\bar{x}d=\bar{y}p\neq 0$, for some  $\bar{x},\bar{y}\in S$. Then  $\bar{x}c\neq 0$ and as  $c\,\lambda\,\bar{x}c,  b\not\hspace{.05mm}\lambda\,\bar{x}c$, so that again,  $Y=0$.\\
\\
On the other hand, if  $b\,\lambda\,c$ so that  $xb=yc\neq 0$ for some $x,y\in S$, and if  $yd\not\hspace{.05mm}\lambda\,p$, then  $d\not\hspace{.05mm}\lambda\,p$ so that  $Y=0$.\\
\\
Conversely, if  $Y=0$, then if  $d\not\hspace{.05mm}\lambda\,p$ we have either  $b\not\hspace{.05mm}\lambda\,c$ (which case  $X=0$)\ or \ $b\,\lambda\,c$ and  $yd\neq 0$. In this case, $p\not\hspace{.05mm}\lambda\,yd$, so that $X=0$. If $d\,\lambda\,p$, then we must have that that $b\not\hspace{.05mm}\lambda\,\bar{x}c$, so that  $b\not\hspace{.05mm}\lambda\,c$ and  again $X=0$. We therefore assume that  $X\neq 0$\ and \ $Y\neq  0$. Then
\[X=[xa,yd][p,q]=[sxa,rq] \ , syd=rp \neq 0 \]
\[Y=[a,b][\bar{x}c,\bar{y}q]=[\bar{s}a,\bar{r}\bar{y}q]  , \bar{s}b=\bar{r}\bar{x}c \neq 0.\]
for some $s,\bar{s} \in S$.\\
\\ 
We have to show that  $X=Y$  i.e. \[wsxa=\bar{w}\bar{s}a\neq 0  , wrq=\bar{w}\bar{r}\bar{y}q\neq 0\]   for some  $ w ,  \bar{w} \in S$. By 0-cancellativity this equivalent to  $ wsx=\bar{w}\bar{s}\neq 0 ,  wr=\bar{w}\bar{r}\bar{y}\neq 0$. 
Since  $ xb \neq 0 ,  sx \neq 0 $  and  $S$ categorical at  0 we have  $sxb \neq 0 \ $also$, \bar{s}b \neq 0$. Hence  $sxb\,\lambda\,\bar{s}b,$ and so there exist  $w,\bar{w} \in S$ such that   $wsxb=\bar{w}\bar{s}b\neq 0$. As   $S$ 0-cancellative, we have  $wsx=\bar{w}\bar{s}\neq 0$. \\
\\
Now, since  $ wsxb=\bar{w}\bar{s}b \neq 0$ and  $\bar{s}b=\bar{r}\bar{x}c\neq 0, xb=yc\neq 0$   we have    $wsyc=\bar{w}\bar{r}\bar{x}c \neq 0$. As  $S$ is 0-cancellative we have  $wsy=\bar{w}\bar{r}\bar{x}\neq 0$, then  $wsyd=\bar{w}\bar{r}\bar{x}d\neq 0$, but  $syd=rp \neq 0$  and  $\bar{x}d=\bar{y}p\neq 0$ so that  $wrp=\bar{w}\bar{r}\bar{y}p \neq 0$. Thus   $wr=\bar{w}\bar{r}\bar{y}\neq 0$ as required. \end{proof}

Let  $[a,b] \in Q^{*}$, by (D) for  $a \in S^*$   there exist  $x \in S$ such that  $xa\neq 0$. Clearly, $[xa,xb]\in Q^*$. There exist $t\in S$ with $txa\neq 0$, so $(tx)a=t(xa)\neq 0$ and as $a\,\mathcal{R^*}\,b$, $(tx)b=t(xb)\neq 0$. Then the following lemma is clear.
\begin{lemma}\label{weneedit}
If  $[a,b] ,  [xa,xb] \in Q^*$,  then  $[a,b]=[xa,xb]$.\end{lemma}
 If  $a \in S^*$, by (D) there exist $x\in S$ such that  $xa \neq 0$. From Lemma~\ref{conc}, we get  $x\,\mathcal{R}^{*}\,xa$. Hence  $[x,xa] \in Q^*$. 
 If  $(y,ya) \in \Sigma^*$, then as $xa\,\lambda\,ya$   there exist  $s,\acute{s} \in S$ with  $sxa=\acute{s}ya\neq 0$ and we have  $sx=\acute{s}y\neq 0$, that is,  $[x,xa]=[y,ya]$. Hence we have got the first part of the following lemma:\\ 
\begin{lemma}\label{embedding}
The mapping  $\theta: S \longrightarrow Q$ defined by   $0\theta=0$  and for  $a\in S^*, \\ a\theta =[x,xa]$ where   $x \in S^*$ with  $xa\neq 0$, is an embedding of  $S$ into  $Q$.\end{lemma}
\begin{proof}If  $a,b\neq 0$ and 
$a\theta =b\theta$, that is,  $[x,xa]=[y,yb]$  for some  $ x , y \in S $  and  $ xa \neq 0 ,  yb \neq 0$, then there exist  $ w ,  \bar{w} \in S $  such that \[wx=\bar{w}y\neq 0 ,  wxa=\bar{w}yb\neq 0. \] Then, $wxa=wxb\neq 0$\ and as  $S$ is 0-cancellative we have  $a=b$ . Thus  $\theta$ is one-one.\\
\\
To show that  $\theta $ is a homomorphism, let  $a,b\in S^*$ and  $a\theta=[s,sa],b\theta=[t,tb]$  where  $sa\neq 0$ and  $tb\neq 0$. \\
\\
Suppose that  $ab=0$. If  $sa\,\lambda\,t$, then  $usa=vt\neq 0$ for some  $u,v \in S$. By categoricity at 0, $usab=vtb\neq 0$, a contradiction. Hence  $sa\,\not\hspace{.05mm}\lambda\,t$ and  $a\theta b\theta=0=(ab)\theta$.\\
 \\
 Assume therefore that  $ab\neq 0$. Let $(ab)\theta=[x,xab]$ where  $xab\neq 0$. By categoricity at 0,  $sab\neq 0$. Hence  $sab\,\lambda\,b\,\lambda\,tb$, so that  $wsab=\bar{w}tb\neq 0$. \\
\\
 Since  $S$ is 0-cancellative  $wsa=\bar{w}t\neq 0$, that is,  $sa\,\lambda\,t$ and we have  $a\theta b\theta\neq 0$. Moreover from  $xa\neq 0$ and  $sa\neq 0$ we have  $sa\,\lambda\,xa$, that is, there exist  $m,n\in S$ such that  $msa=nxa\neq 0$. By cancelling  $a$ we have  $ms=nx\neq 0$ and by categoricity at 0,  $msab=nxab\neq 0$. Thus \[\begin{array}{rcl}a\theta b\theta&=&[s,sa][t,tb]\\ &=&[ws,\bar{w}tb]   \\ &=&[ws,wsab]\\ &=&[s,sab] \quad \mbox{by Lemma~\ref{weneedit} }\\ &=& [x,xab]\\&=&(ab)\theta.\end{array}\]
\end{proof} 
\begin{lemma}\label{regular}
The semigroup  $Q$  is regular. \end{lemma} 
\begin{proof}
Let  $ [a,b] \in Q^{*}$, then since  $[b,a] \in Q^{*}$  we get 

\[\begin{array}{rcl} [a,b][b,a][a,b]&=&[xa,xa][a,b]\quad \mbox{for some}\; x\in S \; \mbox{with} \; xb\neq 0\\
&=&[a,a][a,b]\quad \mbox{by Lemma~\ref{weneedit} }\\
&=&[ya,yb]\quad \mbox{for some}\quad y\in S \; \mbox{with} \; ya\neq 0\\ 
&=&[a,b]\quad \mbox{by Lemma~\ref{weneedit} }.\end{array}\]\end{proof}

For any  $a\in S^{*}$ we have  $[a,a] \in Q^{*}$ and  $[a,a][a,a]=[xa,xa]=[a,a]$ by Lemma~\ref{weneedit}, that is,  $[a,a]$ is idempotent. The next lemma describe the form of  $E(Q)$.
\begin{lemma}\label{idempotents}
\ $E(Q)=\{[a,a], \; a \in S^{*}\}\cup\{0\}$ and forms a semilattice. \end{lemma}
\begin{proof}
Let  $ [a,b] \in E(Q^*)$, then  $ [a,b][a,b]=[a,b]$. Hence  $[xa,yb]=[a,b]$  where  $ xb=ya \neq 0$ . Thus there exists  $  t ,  r \in S^*$  such that  $ txa=ra \neq 0 ,  tyb=rb \neq 0 $. By (B),   $ tx=ty=r \neq 0$, then  $ x=y$ \ and  \ $ a=b$.\\
\\
For  $[a,a],[b,b] \in E(Q^*)$ we have  $[b,b][a,a]=0 \Longleftrightarrow \; a\,\not\hspace{.05mm}\lambda\,b \; \Longleftrightarrow \; [a,a][b,b]=0$ and if  $a\,\lambda\,b$, then  
\[\begin{array}{rcl} 
[a,a][b,b]&=&[\acute{s}a,\acute{t}b]  \quad \mbox{where} \quad \acute{s}a=\acute{t}b\neq 0\; \mbox{for some} \; \acute{s} , \acute{t} \in S  \\
&=&[\acute{t}b,\acute{s}a] \quad \mbox{where} \quad \acute{s}a=\acute{t}b\neq 0 \\
&=&[b,b][a,a].\end{array}\]
 
\end{proof}

We can note easily that  $[b,a]$ is unique, which means that it is the inverse of \ $[a,b]$. \\

\begin{lemma}\label{primitive}
The semigroup  $Q$  is primitive.\end{lemma}
\begin{proof}
Suppose that  $ [a,a] ,  [b,b] \in E(Q^*)$  are such that  $ [a,a]\leq [b,b]$. Then 
\[\begin{array}{rcl} [a,a]&=&[a,a][b,b]\\
&=&[xa,yb] \quad \mbox{for some} \; x,y \in S \; \mbox{where}\ xa=yb\\
&=&[yb,yb]\\ 
&=&[b,b]\quad \mbox{by Lemma~\ref{weneedit} }.  \end{array}\]
\end{proof}

By Lemma~\ref{embedding}  we can regard  $S$ as a subsemigroup of  $Q$. Let  $[a,b]\in Q^*$ and  $a\theta=[x,xa], b\theta=[y,yb]$ where  $xa\neq 0$ and  $yb\neq 0$. Hence \[\begin{array}{rcl} (a\theta)^{-1}(b\theta)&=&[x,xa]^{-1}[y,yb]\\ &=&[xa,x][y,yb]\\ &=&[txa,ryb]\quad \mbox{for some}\; t,r \in S\; \mbox{where}\; tx=ry\neq 0 \\ &=&[txa,txb]\\ &=&[a,b]\quad \mbox{by Lemma ~\ref{weneedit}}.\end{array}\]
Hence   $S$ is a left I-order in  $Q$.
\end{proof}

It is worth pointing out that if  $e\in E(Q^*)$, then  $e=a^{-1}a$ for some  $a\in S^*$. For  $e=a^{-1}b \in E(Q^*)$ as  $a\,\mathcal{R}\,b$ we have  $b=ae$ and  $a=be$. Then it is clear that  $a=b$. \\
\\
A \emph{Brandt semigroup} is a completely 0-simple inverse semigroup. By Theorem II.3.5 in \cite{inverse}  every Brandt semigroup is isomorphic to  $B(G,I)$ for some group  $G$ and non-empty set  $I$ where  $B(G,I)$ is constructed as follows:\\ As a set  $B(G,I)=(I \times G \times I)\cup \{0\},$ the binary operation is defined by

\begin{displaymath}
 (i,a,j)(k,b,l)= \left\{ \begin{array}{ll}
(i,ab,l) & \textrm{if \ $j= k$}\\
0 &  \textrm{else}
\end{array} \right. 
\end{displaymath} 
\[(i,a,j)0=0(i,a,j)=00=0.\]
Note that every Brandt semigroup is a primitive inverse semigroup.

\begin{example}\label{ex:brandt}\cite{GG} Let $H$ be a left order in a group $G$, and let $\mathcal{B}^0=\mathcal{B}^0(G,I)$ be a Brandt semigroup over  $G$
where $|I|\geq 2$. Fix $i\in I$ and let \[S_i=\{ (i,h,j):h\in H, j\in I\}\cup \{ 0\}.\] Then $S_i$ is a straight left I-order in $\mathcal{B}^0$.

To see this, notice that $S_i$ is a subsemigroup, $0=0^{-1}0$, and for any
$(j,g,k)\in \mathcal{B}^0$, we may write $g=a^{-1}b$ where $a,b\in H$ and then
\[(j,g,k)=(i,a,j)^{-1}(i,b,k)\]
where $(i,a,j), (i,b,k)\in S_i$. 

Again, it is easy to see that $S_i$ is not a left order in $\mathcal{B}^0$.
\end{example}

Let  $\{S_{i} : i \in I \}$ be  a family of disjoint semigroups with zero, and put \\ $S^*_i=S\setminus\{0\}$. Let  $ S=\bigcup_{i\in I} S^*_i\cup 0$  with the multiplication 
 \begin{displaymath}
a*b = \left\{ \begin{array}{ll}
ab \ \mbox{if} \ a, b \in S_{i} \ \mbox{for some} \  i \ \mbox{and} \;  ab\neq 0 \quad \mbox{in} \ S_{i}\\
\quad 0 \hspace{2cm} \mbox{else.}
\end{array} \right.\end{displaymath} 
 With this multiplication  $S$ is a semigroup called \emph{0-direct union} of the $S_{i}$. In \cite{clifford}, it is shown that every primitive inverse semigroup with zero is a 0-direct union of Brandt semigroups.
\begin{corollary}\label{maincoro}
A semigroup  $S$  is a left I-order in a Brandt semigroup  $Q$ if and only if  $S$ satisfies that conditions in Theorem ~\ref{maintheorem}  and in addition, 
for all  $a , b \in S^*$  there exist  $c , d \in S$  such that $ca\,\mathcal{R^{*}}\,d\,\lambda\,b$.
 \end{corollary}
 
\begin{proof}Suppose that  $S$  is a left I-order in $Q$  and let  $a,b \in S^{*}$. Since  $Q$ has a single non-zero $\mathcal{D}$-class,  there exists  $q\in Q$ such that  $a\,\mathcal{R}\,q\,\mathcal{L}\,b$ in  $Q$. Let \ $q=c^{-1}d$\ where  $c,d \in S$ and  $c\,\mathcal{R}\,d$. Then  $ca\,\mathcal{R}\,d$ and  $d\,\mathcal{L}\,c^{-1}d\,\mathcal{L}\,b$  that is,  $ca\,\mathcal{R}\,d\,\mathcal{L}\,b$. By Proposition~\ref{Inthenameofgod} $ca\,\mathcal{R^{*}}\,d\,\lambda\,b$.\\
\\
On the other hand, if  $S$ satisfies the given conditions, then we can show that  $Q$ is Brandt. For, if $q=a^{-1}b , p=c^{-1}d \in Q^*$, then  $b,d \in S^*$ so there exist $u,v\in S$ with $ub\,\mathcal{R^*}\,v\,\lambda\,d$. In $Q, ub\,\mathcal{R}\,v\,\mathcal{L}\,d$, so that $b$ and $d$ (and hence $q$ and $p$) lie in the same Brandt subsemigroup of $Q$.
\end{proof}

\begin{lemma}\label{directunion}
Let  $Q=\bigcup_{i\in I} Q_{i}$  be a primitive inverse semigroup where  $Q_{i}$  is a Brandt semigroup. If  $S$ is a left I-order in  $Q$, then  $S$  is a 0-direct union of semigeoups that are left I-orders in the Brandt semigroups  $Q_i$'s.\end{lemma}

\section{Uniqueness}
\label{sec:UBrandt} \index{UBQuotients}
In this section we show that  a semigroup  $S$ has, up to isomomorphism,  at most one  primitive inverse semigroup of left I-quotients.

\begin{definition}\label{defn:lifting} \cite{GG} Let $S$ be a subsemigroup of $Q$ and let $\phi:S\rightarrow
P$ be a morphism from $S$ to a semigroup $P$. If there
is a morphism $\overline{\phi}:Q\rightarrow P$ such that
$\overline{\phi}|_S=\phi$, then we say that
$\phi$ {\em lifts to } $Q$. If $\phi$ lifts to an isomorphism, then
we say that $Q$ and $P$ are {\em isomorphic over $S$}.
\end{definition}
On a straight left I-order semigroup  $S$ in a semigroup  $Q$ we define an ternary relation  $\mathcal{T}^{Q}$ from\cite{GG} on $S$ as follows: \[(a,b,c) \in \mathcal{T}^{Q} \Longleftrightarrow ab^{-1}Q\subseteq c^{-1}Q . \]

Since every left I-order in a primitive inverse semigroup is straight, then we are able to use the following result.

\begin{corollary}\label{cor:iso}\cite{GG} Let $S$ be a straight left I-order
in $Q$ and let $\phi:S\rightarrow P$ be an embedding of $S$
into an inverse semigroup $P$ such that $S\phi$ is a straight left
I-order in $P$. Then $Q$ is isomorphic to $P$ over $S$ if and only if
for any $a,b,c\in S$:
$(i)$ $(a,b)\in \mathcal{R}^Q_S\Leftrightarrow (a\phi,b\phi)\in \mathcal{R}^P_{S\phi}$;
  and \\
$(ii)$ $(a,b,c)\in \mathcal{T}^Q_S\Leftrightarrow (a\phi,b\phi,c\phi)\in
\mathcal{T}^P_{S\phi}$.
\end{corollary}

\begin{theorem}\label{thm:brandhoms} Let $S$ be a left I-order in a primitive inverse 
semigroup $Q$. If $\phi:S\rightarrow T$ is an isomorphism where
$T$ is a left I-order in a primitive inverse semigroup\ $P$, then
$\phi$ lifts to  an isomorphism $\overline{\phi}:Q\rightarrow P$
\end{theorem}

\begin{proof} 

Let $\phi:S\rightarrow T$ be as given. From Proposition~\ref{Inthenameofgod}   $S$ and $T$ both contain 0 and clearly $\varphi$ preserves this. Let $a,b\in S^*$ with $a\,\mathcal{R}\, b$ in $Q$. By Condition (D) there exists $c\in S$ with $ca\neq 0$ and hence $cb\neq 0$. It follows 
 that $(c\phi)(a\phi), (c\phi)(b\phi)$ are non-zero in $P$, so that $a\phi\,\mathcal{R}\, b\phi$ in $P$. Also $\varphi$ preserves $\mathcal{L}$. If $a , b \in S^*$ and $a\,\mathcal{L}\,b$ in $Q$, then $a\,\lambda\,b$ in $S$ so that $ca=db\neq 0$ for some $c,d \in S$. Then $ c\varphi a\varphi = d\varphi b\varphi \neq 0$ so that $a\varphi \,\mathcal{L}\, b\varphi$ in $P$. \\
We now show that $\phi$ preserves $\mathcal{T}^Q_S$. Suppose
therefore that
$a,b,c\in S$ and $ab^{-1}Q
\subseteq c^{-1}Q$. Then either $ab^{-1}=0$, 
or $ab^{-1}\,\mathcal{R}\, c^{-1}$ in $Q$. In the former case, either $a$ or $b$ is 0 or
$a$ and $b$ are not $\mathcal{L}$-related in $Q$ it follows that either $a\varphi$ or $b\varphi$ is 0 or 
$a\phi$ and $b\phi$ are not $\mathcal{L}$-related in $P$, giving
$(a\phi)(b\phi)^{-1}=0$ and so $(a\phi)(b\phi)^{-1}P
\subseteq (c\phi)^{-1}P$. On the
other hand, if $ab^{-1}\neq 0$, then we have $a,b \neq 0$, 
$a\,\mathcal{L}\,b$ and $a\,\mathcal{R}\, c^{-1}$ in $Q$. It follows
that $ca\neq 0$   $a\phi\,\mathcal{L}\, b\phi$ and $(c\phi)(a\phi)
\neq 0$ in $P$. Consequently,
\[0\neq (a\phi)(b\phi)^{-1}P
=(a\phi)P=(c\phi)^{-1}P.\] 

Since $\phi$ (and, dually, $\phi^{-1}$) preserve both $\mathcal{R}$ and $\mathcal{T}$, 
it follows from Corollary~\ref{cor:iso} that
$\phi$ lifts to an isomorphism $\overline{\phi}:
Q\rightarrow P$.

\end{proof}

  The following corollary may be deduced from the previous theorem.
  \begin{corollary}\label{brandtunique}
  If  $Q_{1}\ , Q_{2}$ are primitive inverse semigroups of left I-quotients of a semigroup  $S$, then $Q_{1}\ , Q_{2}$ are isomorphic by an isomorphism which restricts to the identity map on  $S$.\end{corollary} 

\begin{proposition}\label{twosided}
If a semigroup  $S$ has a primitive inverse semigroup  $Q$ of left I-quotients and a primitive inverse semigroup  $\acute{Q}$ of right I-quotients, then  $Q$ and $\acute{Q}$ are both semigroups of I-quotients of $S$, so that $Q\cong\acute{Q}$ by Corollary~\ref{brandtunique}.\end{proposition}
\begin{proof}

We show that $Q$  is a semigroup of right I-quotients of $S$. 
Let $q\in \acute{Q}$. If $q=0$, then $q=\acute{0}0=0\acute{0}$. If $q\in Q^*$, then $q=a^{-1}b$ for some $a,b \in S^*$ where $a\,\mathcal{R}\,b$ in $Q$. Then $a\,\mathcal{R^*}\,b$ in $S$. Pick $c\in S$ with $ca\neq 0$. Then $ca\,\mathcal{R^*}\,cb$ in $\acute{Q}$, so that by the dual of Proposition ~\ref{Inthenameofgod} $ca\,\rho\,cb$ in $S$, that is $cax=cby\neq 0$ for some $ x,y \in S$. As $S$ is 0-cancellative, $ax=by\neq 0$, so that $xy^{-1}=a^{-1}b$ in $Q$ (and $\acute{Q}$). Thus $S$ an I-order in $Q$ and similarly, in $\acute{Q}$.

\end{proof}

   \section{Primitive inverse semigroups of I-quotients}
\label{sec:I-order} \index{Search engines}
 In this section we study the case where a semigroup is both a left and right I-order in a primitive inverse semigroup, that is, an I-order.

\begin{lemma}\label{quotients}
Let \ $S$ have a primitive inverse semigroup  $Q$ of I-quotients. Then\\
$1)$ \ $ \mathcal{R}^{*}=\mathcal{R}\cap(S\times S) =\rho$, \\
$2)$\ $\mathcal{L}^{*}=\mathcal{L}\cap(S\times S)=\lambda$,\\
$3)$\ $ \mathcal{H}^{*}=\mathcal{H}\cap(S\times S)=\tau$. \end{lemma} 
\begin{proof}(1) By Proposition~\ref{Inthenameofgod} $\mathcal{R}\cap(S\times S)=\mathcal{R^*}$. By the dual of Proposition~\ref{Inthenameofgod}, $\mathcal{R}\cap(S\times S)=\rho$. Hence $ \rho=\mathcal{R^*}$.\\
2) \ Similar. \\
3) \ Immediate from 1) and 2).\end{proof}
Since $\mathcal{H}$ is a congruence on any primitive inverse semigroup the following corollary is clear. \\

\begin{corollary}\label{congr}
Let  $Q$ be primitive inverse semigroup of I-quotients of a semigroup  $S$. Then  $\mathcal{H}^{*}$ is a congruence relation on  $S$. \end{corollary}

If  $S$ is an I-order in a primitive inverse semigroup  $Q$, then  $S$ satisfies the conditions in Theorem~\ref{maintheorem} and in addition, the duals  $\acute{(C)}$ and $\acute{(D)}$ of (C) and (D). We can reduce these conditions by using the next lemma.
\begin{lemma}\label{reduce}
Left  $S$ be a left I-order in a primitive inverse semigroup  $Q$ and  suppose that $aS\neq 0$ for all  $a\in S^*$. Then   \[\mathcal{R^{*}} \subseteq \rho \ \mbox{if and only if} \  S \ \mbox{is an I-order in} \ Q.\]\end{lemma}
\begin{proof}
Suppose that  $\mathcal{R^{*}} \subseteq \rho$, by Proposition~\ref{Inthenameofgod}, $\mathcal{R^*} =\rho$, and so $\rho$ is transitive. By  dual of Theorem ~\ref{maintheorem} $S$ is a right I-order in a primitive inverse semigroup $\acute{Q}$. By Proposition~\ref{twosided} $Q \cong\acute{Q}$ and $S$ is an I-order in $Q$. On the other hand, if $S$ is an I-order in $Q$, then by Lemma ~\ref{quotients},  $\mathcal{R^*} =\rho$ as required.
\end{proof}

Now we introduce condition (E) which appeared in \cite{fountain} for a semigroup with zero as follows:\\ (E) \; $a\,\rho\,b$   if and only if  $a=b=0$  or there exist an element $x$ in $S$ such that \[ xa\neq 0 \; \mbox{and} \; xb\neq 0.\]
\begin{lemma}\label{conE}
For a semigroup $S$, the following conditions are equivalent:\\
$(1)$ \ $S$ has a primitive inverse semigroup of I-quotients; \\
$(2)$ \ $S$ is 0-cancellative, categorical at 0, and $S$ satisfies (D), $\acute{(D)}$ , (E) and $\acute{(E)}$.\end{lemma}
\begin{proof} If (1) holds, then by Theorem~\ref{maintheorem} and its dual, we need only to show that $S$ satisfies (E) and its dual. Suppose that $a\,\rho\,b$ and $a,b\neq 0$. By Lemma~\ref{quotients} $a\,\mathcal{R^*}\,b$, and by (D) there is $x\in S$ such that $xa\neq 0$, and so $xb\neq 0$. Conversely, if $xa\neq 0$ and $xb\neq 0$, then using (D), $xa\,\rho\,x$ and $x\,\rho\,xb$. Since $\rho$ is transitive, $xa\,\rho\,xb$, that is, $xat=xbr\neq 0$ for some $t,r\in S$, by cancelling $x$ we have $at=br\neq 0$. Thus $a\,\rho\,b$. Similarly $S$ satisfies $ \acute{(E)}$.\\
\\
If (2) holds, we show that $\lambda$ and $\rho$ are transitive. In order to prove this, we need to show that $\mathcal{R^*}=\rho$ and $\mathcal{L^*}=\lambda$. Let $a\,\mathcal{R^*}\,b$, then either $a=b=0$ or $a,b\neq 0$ and by (D), $xa\neq 0$ for some $x\in S$ and as $a\,\mathcal{R^*}\,b$, then $xb\neq 0$. By (E) we have that $a\,\rho\,b$ so that $\mathcal{R^*}\subseteq \rho$.\\
\\
Conversely, if $a\,\rho\,b$, then either $a=b=0$ (so that $a\,\mathcal{R^*}\,b$) or $ah=bk\neq 0$ for some $h,k \in S$. Suppose now that $u,v \in S^1$ and $ua=va$. If $ua=va\neq 0$, then by categoricity at 0, $uah=vah\neq 0$, so that $ubk=vbk\neq 0$ and 0-cancellativity gives $ub=vb\neq 0$. On the other hand, if $ua=va=0$, then $u,v\in S$ and $ubk=vbk=0$. By categoricity at 0, $ub=vb=0$. Similarly $ub=vb$ implies $ua=va$. Hence $a\,\mathcal{R^*}\,b$.
\end{proof}
We now summaries the result of this Section.
\begin{proposition}\label{conclusionpro}
For a semigroup $S$, the following conditions are equivalent:\\
(1) \ $S$ is an I-order in a primitive inverse semigroup; \\
(2) \ $S$ satisfies conditions (A),(B),(C),(D), $\acute{(C)}$ and $ \acute{(D)}$; \\
(3) \ $S$ satisfies conditions (A),(B),(C),(D), $ \acute{(D)}$ and $\mathcal{R^*}\subseteq \rho$; \\
(4) \ $S$ satisfies conditions (A),(B),(D),$\acute{(D)}$,(E) and $\acute{(E)}$. \end{proposition}
\begin{proof}The equivalence of (1) and (2) follows from Theorem~\ref{maintheorem} and its dual, and Proposition~\ref{twosided}. The equivalence of (1) and (3) is immediate from Theorem~\ref{maintheorem} and its dual, and Lemma~\ref{quotients} and ~\ref{reduce} . Finally, the equivalence of (1) and (4) is gives by Lemma~\ref{conE}.
 \end{proof}

\section{The abundant case} 
A semigroup \ $S$\ is a \emph{left abundant} if each $\mathcal{R^*}$-class contains at least one idempotent. Dually, we can define a \emph{right abundant} semigroup. A semigroup is  \emph{abundant}  if it is  both left and right abundant. If $S$ is left (right) abundant and $E(S)$ is a semilattice, then $S$ is \emph{left} (\emph{right}) \emph{adequate}. Note that in this case, the idempotent in the $\mathcal{R^*}$-class                 ($\mathcal{L^*}$-class) of $a$ is unique. We denote it by $a^+$ ($a^*$). A semigroup which is both left and right adequate will be called an \emph{adequate semigroup}. Fountain \cite{abundant} has generalised the Rees theorem to show that every abundant semigroup in which the non-zero idempotents primitive, is isomorphic to what he calls a \emph{PA-blocked Rees matrix semigroup}. We refer the interested reader to \cite{abundant} for more details. It is clear that if an abundant semigroup is a left I-order in a primitive inverse semigroup, then it is adequate. More than this, it must be ample, as we now explain.\\

We recall that a semigroup  $S$ is a left (right) ample if and only if $S$ is left (right) adequate and satisfies the left (right) ample condition which is;\[(ae)^{+}a=ae\quad (a(ea)^{*}=ea) \quad \mbox{for all} \; a\in S \ \mbox{and} \ e \in E(S).\]
A semigroup is an ample semigroup  if it is both left and right ample. From \cite{Gould}  a semigroup $S$ is left ample if and only if it is embeds in an inverse semigroup $T$ such that $\mathcal{R}\cap (S\times S)=\mathcal{R^*}$. If a left ample semigroup $S$ has a primitive inverse semigroup of left I-quotients $Q$, then for any $a\in S$ we have $a\,\mathcal{R^*}\,a^+$, by Proposition~\ref{Inthenameofgod} $a\,\mathcal{R}\,a^+$, that is $a^+=aa^{-1}$. Hence the following lemma is clear.

\begin{lemma}\label{lapi} Let  $S$  be a left I-order in a primitive inverse semigroup  $Q$. Then the following are equivalent:\\
1)\ $S$\ is  left adequate;\\
2)\ $S$\ is left ample.\end{lemma}
In the next lemma we introduce an equivalent condition for categoricity at 0 for any primitive ample semigroup with zero.

\begin{lemma}\label{amplecate}
Let  $S$ be a primitive  ample semigroup with zero. Then the following are equivalent\\ $(i)$\ $S$ is categorical at 0; \\  $(ii)$\ $a^*=b^+\ \Longleftrightarrow \ ab\neq 0$ for  $a$ and  $b$ in $S^*$. \end{lemma}
\begin{proof}
 $(i)\Longrightarrow (ii)$ \  Let $a,b \in S$. If  $ab\neq 0$, then $aa^*b^+b\neq 0$, so $a^*b^+\neq 0$ and so by primitivity $a^*=b^+$. Conversely, if $a^*=b^+$, then $aa^*\neq 0, a^*b=b^+b\neq 0$, so by categoricity at 0, $ab=aa^*b\neq 0$.\\  $(ii) \Longrightarrow(i)$ Suppose that $ab\neq 0 , bc\neq 0$ where  $a,b,c \in S$, then  $a^*=b^+$ and  $b^*=c^+$. Hence  $b^+(bc)\neq 0$ gives  $b^+=(b^+)^*=(bc)^+$, but  $b^+=a^*$, that is,  $a^*=(bc)^+$. By assumption  $abc\neq 0$.\end{proof}

We can offer some simplification of Theorem~\ref{maintheorem} in the case that $S$ is adequate.
\begin{proposition}\label{jsimple}\cite[Proposition 5.5]{abundant}
For a semigroup $S$ with zero, the following conditions are equivalent:\\
1) \ $S$ is categorical at  0, 0-cancellative and satisfies:\\
for each element  $a$ of  $S$, there is an element  $e$ of  $S$ such that  $ea=a$ and an element $f$ of  $S$ such that  $af=a$\: ....................................(*)\\
2) \ $S$ is a primitive adequate semigroup.\\
3) \ $S$ is isomorphic to PA-blocked Rees $I\times I$ matrix semigroup  $\mathcal{M}(M_{\alpha\beta};I,I,\Gamma;P)$ where the sandwich matrix  $P$ is diagonal and  $p_{ii}=e_{\alpha}$ for each  $i\in I_{\alpha}, \alpha \in \Gamma$.\end{proposition} 
From the above lemma and Theorem~\ref{maintheorem}, the following lemma is clear.
\begin{lemma}\label{abundantleftprimitivie}
For a semigroup $S$ with zero, the following conditions are equivalent: \\
1) \ $S$ is abundant and left I-order in a primitive inverse semigroup $Q$; \\
2) \ $S$ is a primitive adequate semigroup and $\lambda$ is transitive;\\
3) \ $S$ is 0-cancellative, categorical at 0, $S$ satisfies (*)and $\lambda$ is transitive. \end{lemma}
  
In the two-sided case we have the following.
\begin{lemma}\label{abundantmain}
For a semigroup with zero, the following conditions are equivalence:\\
1) \ $S$ is abundant and an I-order in a primitive inverse semigroup $Q$;\\
2) \ $S$ is primitive adequate and $\lambda , \rho$ are transitive;\\
3) \ $S$ is 0-cancellative, categorical at 0, satisfies (*) and $\lambda , \rho$ are transitive. \end{lemma} 
 
	We cannot use any thing in \cite{GG}, because we did not talk about zero in that paper or we should add it or rather we should cancel the following Prop. 
  
\begin{proposition}\label{unionofr}
Let  $S$ be a left ample semigroup and left I-order in a primitive inverse semigroup  $Q$. If  $S$ is a union of  $\mathcal{R}$-classes of  $Q$, then  $Q \cong \Sigma(S)$.\end{proposition}
\begin{proof}By Lemma~\ref{brandtunique}, it is enough to show that $\Sigma(S)$ is primitive, since by Theorem 3.7 and Lemma 3.6 of \cite{GG}, $S$ is a left I-order in $\Sigma(S)$. Let $0\neq e\leq f$ in $\Sigma(S)$, then $e=a^{-1}a$ and $f=b^{-1}b$ for some $a,b \in S$ where $e,f \in E(\Sigma(S))$. We have $0\neq e=ef$, so that $ab^{-1}\neq 0$ and $ab^{-1}=c^{-1}d$ for some $c,d \in S$ with $c\,\mathcal{R^*}\,d$ in $S$ and so $c\,\mathcal{R}\,d$ in $\Sigma(S)$. Then by Lemma 2.6 of \cite{GG}, $ca=db\neq 0$, so that in $Q$, $a\,\mathcal{L}\,b$ and so  $a\,\mathcal{L}\,b$  in $S$ by Lemma 3.6 of \cite{GG} therefore in $\Sigma(S)$. Hence $e=f$. \end{proof}

\end{document}